\documentclass[reqno]{amsart}

\usepackage{hyperref}
\usepackage{mathtools}
\usepackage{mathrsfs}
\usepackage{amsmath}
\usepackage{bm}

\newtheorem{theorem}{Theorem}[section]
\newtheorem{lemma}[theorem]{Lemma}

\newtheorem{question}[theorem]{Question}
\newtheorem{problem}[theorem]{Problem}

\theoremstyle{definition}
\newtheorem{remark}[theorem]{Remark}

\usepackage{amssymb}
\renewcommand{\leq}{\leqslant}
\renewcommand{\geq}{\geqslant}

\def\N{\text N}

\def\SL{\text{SL}}
\def\PSL{\mathrm{PSL}}

\def\P{\mathbf P}

\def\Aut{\text{Aut}}
\def\Out{\text{Out}}

\newcommand{\gen}[1]{\ensuremath{\langle #1\rangle}}

\usepackage{accents}

\makeatletter
\def\blfootnote{\xdef\@thefnmark{}\@footnotetext}
\makeatother

\begin{document}

\title{Connected components in the invariably generating graph of a finite group}

\author{Daniele Garzoni}
\address{Daniele Garzoni, Dipartimento di Matematica ``Tullio Levi-Civita'', Universit\`a degli Studi di Padova, Padova, Italy}
\email{daniele.garzoni@phd.unipd.it}

\maketitle

\begin{abstract}
   We prove that the invariably generating graph of a finite group can have an arbitrarily large number of connected components with at least two vertices.
\end{abstract}

\section{Introduction}
Given a finite group $G$ and a set $X=\{C_1, \ldots, C_t\}$ of conjugacy classes of $G$, we say that $X$ \textit{invariably generates} $G$ if $\gen{x_1, \ldots, x_t}=G$ for every $x_1\in C_1, \ldots, x_t\in C_t$. We write in this case $\gen{X}_I=G$. This concept was introduced by Dixon \cite{Dix92} with motivations from computational Galois theory, and has been widely studied in recent years.

In \cite{Gar}, the following definition was given. For a finite group $G$, the \textit{invariably generating graph} $\Lambda(G)$ of $G$ is the undirected graph whose vertices are the conjugacy classes of $G$ different from $\{1\}$, and two vertices $C$ and $D$ are adjacent if $\gen{C,D}_I=G$. If $G$ is not invariably $2$-generated, $\Lambda(G)$ is the empty graph. Even when $G$ is invariably $2$-generated, the graph $\Lambda(G)$ can have isolated vertices (e.g., when $G$ is not cyclic, the classes contained in the Frattini subgroup); define $\Lambda^+(G)$ as the graph obtained by removing the isolated vertices of $\Lambda(G)$. 
In this paper we prove the following result.

\begin{theorem}
\label{main_theorem}
For every positive integer $n$, there exists a finite group $G$ such that $\Lambda^+(G)$ has more than $n$ connected components.
\end{theorem}

Theorem \ref{main_theorem} should be seen in comparison to the analogous graph for the case of usual generation; see Subsection \ref{comparison_generating_graph}.

In the proof of Theorem \ref{main_theorem}, we use $G=S^\beta$, where $S=\PSL_2(q)$ and $\beta=\beta(S)$ is the largest integer for which $S^\beta$ is invariably $2$-generated.


The crucial observation about $S=\PSL_2(q)$ is that the graph $\Lambda^+(S)$ is bipartite, which follows from the fact that $S$ admits a $2$-covering (Lemma \ref{lemma_bipartite}). See Section \ref{section_comments} for further comments in this direction, related to the clique number and the chromatic number of $\Lambda^+(S)$, in case $S$ is a nonabelian finite simple group. 

In order to estimate the number of connected components of $\Lambda^+(S^\beta)$, 
we will need the following result. 

\begin{theorem}
\label{main_PSL2}
Let $S=\emph{PSL}_2(q)$, and let $C_1, C_2$ be conjugacy classes of $S$ chosen uniformly at random. Then
\[
\P(\gen{C_1, C_2}_I=S)=1/2+O(1/q).
\]
\end{theorem}
The proof of Theorem \ref{main_PSL2} is straightforward, since the subgroups and conjugacy classes of $\PSL_2(q)$ are known very explicitly (and one can be much more precise about the error term).

For our application to Theorem \ref{main_theorem}, we could somewhat shorten the proof by imposing some restrictions on $q$ (e.g., by requiring $q$ prime). However, we prefer to state and prove Theorem \ref{main_PSL2} in general. Indeed, we find it interesting that the asymptotic behaviour of $\P(\gen{C_1, C_2}_I=S)$ is equal to the asymptotic behaviour of $\P(\gen{x_1^S,x_2^S}_I=S)$, where $x_1, x_2\in S$ are random \textit{elements}. For this, see \cite[Subsection 6.1]{GarMcKem}.


\subsection{Some context on $\Lambda(G)$}
\label{context_invgen}
Guralnick--Malle \cite{GM} and Kantor--Lubotzky--Shalev \cite{KLS} independently proved that every finite simple group $S$ is invariably $2$-generated, so that $\Lambda(S)$ is not the empty graph. In \cite{Gar}, the author studied the graph $\Lambda(G)$, for the case where $G$ is an alternating or symmetric group. He proved that $\Lambda^+(G)$ is connected with diameter at most $6$ in these cases (with the exception of $S_6$, which is not invariably $2$-generated).

At present, it is not known whether $\Lambda^+(S)$ is connected if $S$ is nonabelian simple (see \cite[Question 1.6]{Gar}).

In \cite{Gar}, an analogous graph, denoted by $\Lambda_e(G)$, was defined. Its vertices are the nontrivial \textit{elements} of $G$, and two vertices are adjacent if the corresponding classes invariably generate $G$. In \cite{GarMcKem}, it was proved that, if $S$ is nonabelian simple, then either $S$ belongs to three explicit families of examples, or the proportion of isolated vertices of $\Lambda_e(S)$ tends to zero as $|S|\rightarrow \infty$. See \cite[Corollary 1.6]{GarMcKem} for a precise statement and further comments.

\subsection{Comparison to usual generation}
\label{comparison_generating_graph}

For a finite group $G$, the \textit{generating graph} $\Gamma(G)$ of $G$ is the undirected graph whose vertices are the nonidentity elements of $G$, with vertices $x$ and $y$ adjacent if $\gen{x,y}=G$. This graph 
has been intensively studied in the last two decades; see Burness \cite{Bursurvey} and Lucchini--Mar\'oti \cite{LucMar} for many results in this context.

Again, the graph $\Gamma(G)$ can have isolated vertices, and we consider the graph $\Gamma^+(G)$ obtained by removing the isolated vertices of $\Gamma(G)$.

It is known that $\Gamma^+(G)$ is connected in several cases. For instance, Burness--Guralnick--Harper \cite{BGH} showed that, if $G$ is a finite group such that every proper quotient of $G$ is cyclic, then $\Gamma(G)=\Gamma^+(G)$ and $\Gamma(G)$ is connected with diameter at most $2$ (the special case of $G$ simple was proved by Guralnick--Kantor \cite{GK} and Breuer--Guralnick--Kantor \cite{BGK}).

We will recall other results of the same flavour, proved by Crestani--Lucchini \cite{CreLuc, CreLuc2}, in Section \ref{section_comments}.



On the other hand, no example of $G$ is known for which $\Gamma^+(G)$ is disconnected. In fact, it is believed that there should be no such example; see Acciarri--Lucchini \cite[comments after Corollary 2.7]{acciarrilucchini} for a more general conjecture, which would imply that $\Gamma^+(G)$ is connected for every $2$-generated finite group $G$. 

This determines a sharp difference with respect to Theorem \ref{main_theorem}. We note that this difference does not occur for nilpotent groups. Indeed, in a finite nilpotent group every maximal subgroup is normal, therefore the concepts of generation and invariable generation coincide. See Harper--Lucchini \cite{HarLuc} for results on the generating graph of finite nilpotent groups.



\subsection*{Acknowledgements} The author thanks Andrea Lucchini for a discussion about the literature on the generating graph of a finite group, and Daniela Bubboloni for a discussion about the literature on normal coverings of finite simple groups.

\section{Proof of Theorems \ref{main_theorem} and \ref{main_PSL2}}
\label{section_proofs}

\subsection{Direct powers of finite simple groups} Throughout this subsection, $S$ denotes a nonabelian finite simple group. We review some properties of invariable generation of direct powers of $S$, which reflect some interesting properties of the corresponding invariably generating graphs. The key tool is an elementary criterion due to Kantor and Lubotzky \cite{KLu}, which we recall.

Denote by $\Psi_2(S)$ the set of all pairs $(C_1, C_2)$, where $C_i$ is a conjugacy class of $S$, and $\gen{C_1, C_2}_I=S$. 
As recalled in Subsection \ref{context_invgen}, it is known that $\Psi_2(S) \neq \varnothing$.

Now let $t$ be a positive integer, and let $C$ and $D$ be conjugacy classes of $S^t$, with $C=C_1\times \cdots \times C_t$ and $D=D_1\times \cdots \times D_t$ (and $C_i$ and $D_i$ are conjugacy classes of $S$). Consider the matrix
\[
A_{C,D}=
  \begin{pmatrix}
    C_1 & C_2 & \cdots & C_t \\
   D_1 & D_2 & \cdots & D_t
  \end{pmatrix}.
\]

\begin{lemma} 
\label{lemma_elementary_criterion}
We have that $\gen{C,D}_I=S^t$ if and only if the following conditions are both satisfied:

\begin{itemize}
    \item[(i)] Each column of $A_{C,D}$ belongs to $\Psi_2(S)$, and
    \item[(ii)] No two columns of $A_{C,D}$ lie in the same orbit for the diagonal action of $\emph{Aut} (S)$ on $\Psi_2(S)$.
\end{itemize}
\end{lemma}

\begin{proof}
See \cite[Proposition 6]{KLu}, and also \cite[Lemma 20]{DetLuc}.
\end{proof}

Let $\beta=\beta(S)$ be the largest integer for which $S^\beta$ is invariably $2$-generated. Lemma \ref{lemma_elementary_criterion} implies that $\beta(S)$ is equal to the number of orbits for the diagonal action of $\Aut (S)$ on $\Psi_2(S)$. We note the following fact.

\begin{lemma}
\label{lemma_chain_inequalities}
We have
\[
\frac{|\Psi_2(S)|}{|\emph{Out}(S)|}\leq \beta(S) \leq |\Psi_2(S)|.
\]
\end{lemma}

\begin{proof}
The second inequality is clear, and the first follows from the fact that $\text{Inn}(S) \cong S$ acts trivially in the relevant action, hence each orbit has size at most $|\Out(S)|$.
\end{proof}

If $S$ is a nonabelian finite simple group, then $\Out(S)$ is ``very'' small. For instance, $|\Out(S)|=O(\log_2|S|)$ (see e.g. \cite[Tables 5.1.A--5.1.C]{kleidmanliebeck} for the exact value of $|\Out(S)|$ in each case). We expect $|\Out(S)|$ to be much smaller than $|\Psi_2(S)|$ for every sufficiently large nonabelian finite simple group $S$. Therefore, $|\Psi_2(S)|$ should be, in some sense, a good approximation for $\beta(S)$.

\begin{remark}
Let us compare Lemma \ref{lemma_chain_inequalities} with the analogous problem for classical generation.

Let $\delta=\delta(S)$ be the largest integer for which $S^\delta$ is $2$-generated. In contrast to the situation for $\beta(S)$, there is an exact formula for $\delta(S)$, namely, $\delta(S)=\phi_2(S)/|\Aut(S)|$, where $\phi_2(S)$ denotes the number of ordered pairs $(x,y)\in S^2$ such that $\gen{x,y}=S$. (This goes back to Hall \cite{Hal} in the 1930s and has been widely used.) 
The difference is that the diagonal action of $\Aut(S)$ on the set of generating pairs of elements is semiregular (i.e., only the identity fixes a generating pair), while this need not be the case for the action of $\Aut(S)$ on the set of invariable generating pairs of conjugacy classes.
\end{remark}

Lemma \ref{lemma_elementary_criterion} describes quite precisely the graph $\Lambda^+(S^\beta)$. Indeed, any arc in the graph is obtained as follows (and only in this way). Construct a $2\times \beta$ matrix, in which the columns form a set of representatives for the $\Aut(S)$-orbits on $\Psi_2(S)$. Then the first row is adjacent to the second row in $\Lambda^+(S^\beta)$ (here we are identifying a conjugacy class $C_1\times \cdots \times C_\beta$ of $S^\beta$ with a row vector $(C_1, \ldots, C_\beta)$). Since $\Aut(S^\beta)\cong \Aut(S)\wr \text{Sym}(\beta)$ acts by automorphisms on $\Lambda^+(S^\beta)$, we also see that $\Lambda^+(S^\beta)$ is arc-transitive.

\subsection{The case $S=\text{PSL}_2(q)$} Throughout this subsection, we let $S=\PSL_2(q)$, where $q\geq 4$ is a power of the prime $p$, and we let $d=(2,q-1)$.

The maximal subgroups of $S$ are known by work of Dickson \cite{dicksonold} (see also \cite[Table 8.1]{bray_holt_roneydougal}). For the reader's convenience, we report here the list, up to conjugacy. 

\begin{lemma}
\label{maximal_subgroups_psl2}
The maximal subgroups of $S$ are the following, up to conjugacy:

\begin{itemize}
    \item[$\diamond$] A subgroup $B$ of order $q(q-1)/d$ (the stabilizer of a $1$-space).
    \item[$\diamond$] A dihedral group of order $2(q-1)/d$, for $q$ even or $q\geq 13$.
   \item[$\diamond$] A dihedral group of order $2(q+1)/d$, for $q$ even or $q\neq 7,9$.
    \item[$\diamond$] $\PSL_2(q_0)$ where $q=q_0^r$, $q_0\neq 2$ and $r$ is an odd prime.
    \item[$\diamond$] $\mathrm{PGL}_2(q_0)$ where $q=q_0^2$ and $q_0\neq 2$ (two classes if $q$ is odd).
    \item[$\diamond$] $A_4$ for $q=p\equiv \pm3, 5, \pm 13 \pmod{40}$; $S_4$ for $q=p\equiv \pm 1\pmod 8$ (two classes); $A_5$ for either $q=p\equiv \pm 1 \pmod{10}$, or $q=p^2$ and $p\equiv \pm 3 \pmod{10}$ (two classes).
\end{itemize}
\end{lemma}

For completeness, we have reported the exact conditions on $q$ under which the relevant items yield maximal subgroups of $S$, although these are largely irrelevant for our purposes. For instance, the conditions on $q$ in the last item can safely be ignored (except possibly for the purpose of Remark \ref{rmk: graph_lambda}).


We recall other well known facts.

\begin{lemma}
\label{lemma_conjugacy_classes_PSL}
The following hold.

\begin{enumerate}
    \item $S$ contains a unique conjugacy class of involutions and, for $p$ odd, two conjugacy classes of elements of order $p$.
    \item Let $\ell\geq 3$ be a divisor of  $(q\pm 1)/d$. There are $\phi(\ell)/2$ conjugacy classes of elements of order $\ell$ in $S$, where $\phi$ is Euler's totient function.
    \item The number of conjugacy classes of $S$ is $(q+4d-3)/d$.
\end{enumerate}
\end{lemma}

\begin{proof}
We sketch a proof. (1) Assume first $q$ is odd, and let us deal with involutions. Let $\varepsilon=1$ if $q\equiv 1\pmod 4$, and $\varepsilon=-1$ otherwise. By explicit matrix computation, we find that the number of involutions of $S$ is $q(q+\varepsilon)/2$. This coincides with the number of conjugates of a dihedral subgroup of order $q-\varepsilon$, so we deduce that all involutions of $S$ are conjugate. Next we deal with elements of order $p$ (including the case $q$ even). The image in $S$ of the subgroup of $\SL_2(q)$ consisting of the upper unitriangular matrices is a Sylow $p$-subgroup $P$ of $S$. This is contained in a subgroup $B$, consisting of the image of the upper triangular matrices of $\SL_2(q)$. Let $1\neq x\in P$. We verify that $x^S\cap P=x^B$, and we compute that conjugating $x$ by elements of $B$ can only multiply the upper-right entry by every nonzero square of $\mathbf F_q$. This proves (1).

(2) 
Let $\ell\geq 3$ be a divisor of  $(q\pm 1)/d$. All cyclic subgroups of order $\ell$ are conjugate in $S$. Assume $x\in S$ has order $\ell$. We have that $\N_S(\gen x)$ is dihedral of order $2(q\pm 1)/d$, from which $x^S\cap \gen x=\{x,x^{-1}\}$. This proves (2).

(3) An element of $S$ has either order $p$, or order dividing $(q\pm 1)/d$. Then (3) follows from (1), (2), and the formula
\[
\sum_{\ell \mid (q\pm 1)/d} \frac{\phi(\ell)}{2}=\frac{q\pm 1}{2d}.
\]
The statement is proved.
\end{proof}

The following lemma represents the main observation regarding $S=\PSL_2(q)$.

\begin{lemma}
\label{lemma_bipartite}
The graph $\Lambda^+(S)$ is bipartite.
\end{lemma}

\begin{proof}
Let $H$ be a dihedral subgroup of $S$ of order $2(q+1)/d$, and $K$ be a Borel subgroup of order $q(q-1)/d$. It is well known (see e.g. \cite{BubboLucid}) that $\{H,K\}$ is a $2$-covering of $S$, that is,
\[
S=\bigcup_{g\in S} H^g \cup \bigcup_{g\in S} K^g.
\]
Write for convenience $\widetilde H=\cup_{g\in S}H^g$ and $\widetilde K=\cup_{g\in S}K^g$. A conjugacy class contained in $\widetilde H \cap \widetilde K$ is isolated in $\Lambda(S)$, and a class contained in $\widetilde H \setminus \widetilde K$ can be adjacent in $\Lambda^+(S)$ only to a class contained in $\widetilde K \setminus \widetilde H$. This gives a partition of $\Lambda^+(S)$ into two parts. 
\end{proof}

We refer to Subsection \ref{subsection_bipartite} for a discussion on the topic of $2$-coverings of finite simple groups.

\begin{remark}
\label{rmk: graph_lambda}
Using Lemmas \ref{maximal_subgroups_psl2} and \ref{lemma_conjugacy_classes_PSL}, it is not hard to show that $\Lambda^+(S)$ is connected with diameter at most $3$. Now, let us briefly comment on the graph $\Lambda(S)$. How does it differ from $\Lambda^+(S)$? It is straightforward to prove the following:

\begin{itemize}
    \item[(i)] If $q=7$, then $\Lambda(S)$ has $1$ isolated vertex: the class of elements of order $3$.
    \item[(ii)] If $q=9$, then $\Lambda(S)$ has $3$ isolated vertices: the two classes of elements order $3$, and the class of involutions.
    \item[(iii)] If $q\neq 9$, and $q$ is even or $q\equiv 1 \pmod 4$ or $q\neq p$, then $\Lambda(S)$ has $1$ isolated vertex: the class of involutions. 
    \item[(iv)] If $q\neq 7$ and $q=p\equiv 3 \pmod 4$, then $\Lambda(S)$ has no isolated vertices, therefore $\Lambda(S)=\Lambda^+(S)$.
    \end{itemize}

In particular, if $q\neq 9$ then $\Lambda(S)$ has at most one isolated vertex. This is in contrast to the case of alternating or symmetric groups $G$, in which case $\Lambda(G)$ can have an arbitrarily large number of isolated vertices (see \cite[Theorem 1.2]{Gar}).
\end{remark}

Let $\mathscr P_1$ and $\mathscr P_2$ be the parts of $\Lambda^+(S)$ given in the proof of Lemma \ref{lemma_bipartite}. We note that, for every conjugacy class $C$ of $S$ and for every $\sigma \in \Aut(S)$, $\{C,C^\sigma\}$ does not invariably generate $S$. (A way to see this is that the sets $\widetilde H$ and $\widetilde K$ from the proof of Lemma \ref{lemma_bipartite} are preserved by every automorphism of $S$.) In particular, for every $(C_1, C_2) \in \Psi_2(S)$, $(C_1, C_2)$ and $(C_2, C_1)$ belong to different $\Aut(S)$-orbits. We also note that the parts $\mathscr P_1$ and $\mathscr P_2$ are invariant under the action of $\Aut(S)$. We deduce the following

\begin{lemma}
\label{delta_even}
$\beta=\beta(S)$ is even, and for each vertex $C=C_1\times \cdots \times C_\beta$ of $\Lambda^+(S^\beta)$, there exists a subset $\Omega=\Omega(C)$ of $\{1, \ldots, \beta\}$ of size $\beta/2$ such that for every $i \in \Omega$, $C_i\in \mathscr P_1$, and for every $i \not\in \Omega$, $C_i\in \mathscr P_2$.
\end{lemma}

We can finally prove the key result.

\begin{theorem}
\label{theorem_key_psl}
The graph $\Lambda^+(S^\beta)$ has at least $\frac{1}{2}\cdot \binom{\beta}{\beta/2}$ connected components.
\end{theorem}

\begin{proof}
For a vertex $C=C_1\times \cdots \times C_\beta$ of $\Lambda^+(S)$, let $\Omega(C)$ be the set from Lemma \ref{delta_even}. Then, $C$ can be adjacent only to vertices $D$ such that $\Omega(D)=\{1, \ldots, \beta\}\setminus \Omega(C)$. In particular, the number of connected components of $\Lambda^+(S^\beta)$ is at least half the number of $\beta/2$-subsets of $\{1, \ldots, \beta\}$, which proves the statement.
\end{proof}

It is not difficult to establish that $\beta(S)$ tends to infinity as $|S|\rightarrow \infty$ (that is, $q\rightarrow \infty$), thereby proving Theorem \ref{main_theorem}. In the next subsection we will obtain a better estimate for $\beta(S)$.

\subsection{Bounds} We want to estimate $\beta(S)$, where $S=\PSL_2(q)$. We will find the asymptotic behaviour of $|\Psi_2(S)|$, and then apply Lemma \ref{lemma_chain_inequalities}. 
In the following, $f=O(g)$ means that $|f|\leq Cg$ for some constant $C$ (so $f$ might also be negative). 

\begin{theorem}
\label{theorem_estimate_is}
Let $S=\emph{PSL}_2(q)$ and $d=(2,q-1)$. We have
\[
|\Psi_2(S)|=\frac{q^2}{2d^2}+O(q).
\]
(For $q$ odd the first term of the expression is not an integer, but still the statement makes sense.)
\end{theorem}

\begin{proof}
In this proof, when we say that a conjugacy class $C$ \textit{intersects} a subgroup $H$, we mean that $C\cap H\neq \varnothing$.

We need to count the pairs of conjugacy classes $(C_1, C_2)$ which invariably generate $S$. We ignore the pairs where either $C_1$ or $C_2$ comprises elements of order $p$, or order at most $2$. By Lemma \ref{lemma_conjugacy_classes_PSL}, the number of these pairs is $O(q)$.

By this choice, up to swapping the indices, $C_1$ intersects a cyclic subgroup of order $(q-1)/d$, and $C_2$ intersects a cyclic subgroup of order $(q+1)/d$. Given $C_1$ and $C_2$ with this property, by Lemma \ref{maximal_subgroups_psl2}, we see that $C_1$ and $C_2$ invariably generate $S$ unless one of the following occurs:
\begin{itemize}
    \item[(i)] $C_1$ and $C_2$ intersect a subgroup isomorphic to $A_4$, $S_4$ or $A_5$, and there are at most four conjugacy classes of such subgroups.
    \item[(ii)] $C_1$ and $C_2$ intersect a maximal subgroup conjugate to $\text{PSL}_2(q^{1/r})$ where $r$ is an odd prime (and $q$ is an $r$-th power).
\end{itemize}
(Note that in (ii) we are excluding the subfield subgroups $\text{PGL}_2(q^{1/2})$. Indeed, any class of elements of $\text{PGL}_2(q^{1/2})$ of order prime to $q$ intersects a cyclic subgroup of $S$ order $(q-1)/d$; and this cannot occur for $C_2$.) 
Clearly there are $O(1)$ possibilities for $(C_1,C_2)$ satisfying (i). The number of conjugacy classes of $\PSL_2(q^{1/r})$ is $O(q^{1/r})$; therefore, for fixed $r$, the number of possibilities for the pair $(C_1, C_2)$ satisfying (ii) is $O(q^{2/r})$. Summing over the odd prime divisors of $\log_p q$, we see that there are $O(q^{2/3})$ possibilities in (ii), noting that there are at most $\log_2 \log_2 q$ possibilities for $r$.

Using Lemma \ref{lemma_conjugacy_classes_PSL}, we get the following formula for $|\Psi_2(S)|$ (the factor $2$ at the beginning comes from the fact that we may also have $C_1$ intersecting a cyclic subgroup of order $(q+1)/d$, and $C_2$ intersecting a cyclic subgroup of order $(q-1)/d$):
\begin{align*}
|\Psi_2(S)|&=2\cdot \sum_{\substack{\ell_1 \mid (q-1)/d \\ \ell_2 \mid (q+1)/d}}\frac{\phi(\ell_1)}{2}\frac{\phi(\ell_2)}{2} + O(q) \\ &= \frac{q^2-1}{2d^2}+O(q)
=\frac{q^2}{2d^2} + O(q).
\end{align*}
Here we used the formula $\sum_{\ell \mid n}\phi(\ell)=n$, and the fact that $\phi(\ell_1)\phi(\ell_2)=\phi(\ell_1 \ell_2)$ for coprime $\ell_1$ and $\ell_2$.
\end{proof}
We can rephrase Theorem \ref{theorem_estimate_is} in probabilistic language, that is, we can prove Theorem \ref{main_PSL2}.

\begin{proof}[Proof of Theorem \ref{main_PSL2}]
The statement follows from Lemma \ref{lemma_conjugacy_classes_PSL}(3) and Theorem \ref{theorem_estimate_is}.
\end{proof}

Finally we can estimate $\beta(S)$ and get a lower bound for the number of connected components of $\Lambda^+(S^\beta)$, thereby proving Theorem \ref{main_theorem}. For aesthetic reasons, we denote $2^a=\text{exp}_2\{a\}$, and the symbol $o(1)$ is understood with respect to the limit $q\rightarrow \infty$.

Let $N(S)$ denote the number of connected components of $\Lambda^+(S^\beta)$. 

\begin{theorem}
\label{theorem_lower_bound}
If $S=\emph{PSL}_2(q)$, then
\begin{align}
&q^{2-o(1)} \leq \beta(S) \leq \frac{q^2}{2}+O(q)  \\
&N(S) \geq \emph{exp}_2\left\{q^{2-o(1)}\right\}
\end{align}

\end{theorem}

\begin{proof}
Item (1) follows from Lemma \ref{lemma_chain_inequalities}, the fact that $|\Out(S)|\leq 2\log_2 q$, and Theorem \ref{theorem_estimate_is}.
By Theorem \ref{theorem_key_psl}, Stirling's approximation and (1), we get
\begin{align*}
    N(S)\geq \frac{1}{2}\cdot \binom{\beta}{\beta/2} &=(1+o(1))\cdot \frac{2^\beta}{(2\pi\beta)^{1/2}} \\
    &\geq \text{exp}_2\left\{q^{2-o(1)}\right\},
\end{align*}
and the proof is complete.
\end{proof}

\section{Further comments}
\label{section_comments}
Recall that $\Gamma^+(G)$ is the graph obtained by removing the isolated vertices from the generating graph $\Gamma(G)$ of $G$. Crestani--Lucchini \cite{CreLuc2} showed that, if $G$ is a $2$-generated direct power of a nonabelian finite simple group, then $\Gamma^+(G)$ is connected.

In particular, Theorem \ref{theorem_lower_bound} shows that the result of \cite{CreLuc2} does not hold for invariable generation. Nevertheless, a combinatorial proof along the lines of \cite[Theorem 3.1]{CreLuc2} might be feasible in order to show the following: If a finite simple group $S$ is such that $\Lambda^+(S)$ is connected and not bipartite, then $\Lambda^+(S^t)$ is connected for every $t\leq \beta(S)$. At present, the connectedness of $\Lambda^+(S)$ for $S$ simple is essentially only known for alternating groups \cite{Gar}, which somewhat limits the applications of such a result.

We also remark that we are currently unable to construct examples of soluble groups $G$ for which $\Lambda^+(G)$ is disconnected.

\begin{question}
\label{question_soluble}
Let $G$ be a finite soluble group which is invariably $2$-generated. Is the graph $\Lambda^+(G)$ connected?
\end{question}
Crestani--Lucchini \cite{CreLuc} showed that this is true for the graph $\Gamma^+(G)$ (and, in particular, Question \ref{question_soluble} has a positive answer for nilpotent groups).

\subsection{$\Lambda^+(S)$ bipartite}
\label{subsection_bipartite}
For the proof of Theorem \ref{theorem_lower_bound}, the only important property of $S=\PSL_2(q)$ is that the graph $\Lambda^+(S)$ is bipartite, which follows from the fact that $S$ admits a $2$-covering (see the proof of Lemma \ref{lemma_bipartite}).

The $2$-coverings of the finite simple groups have been well studied; see Bubboloni \cite{Bubbo}, Bubboloni--Lucido \cite{BubboLucid}, Bubboloni--Lucido--Weigel \cite{BLW1, BLW2}, Pellegrini \cite{Pel}. In particular, all finite simple groups admitting a $2$-covering are known, except for some classical groups in small dimension. 


We have the following clear implications:
\begin{equation}
\label{eq:implications}
\text{$S$ admits a $2$-covering} \implies \text{$\Lambda^+(S)$ is bipartite} \implies \text{$\Lambda^+(S)$ has no triangles}
\end{equation}
(These implications are a special case of the inequalities in \eqref{eq:inequalities} below.) The reverse of the first implication in \eqref{eq:implications} does not necessarily hold. For instance, $A_9$ does not admit a $2$-covering (it was proved in \cite{Bubbo} that $A_n$ admits a $2$-covering if and only if $4\leq n\leq 8$). On the other hand, it is not difficult to show that $\Lambda^+(A_9)$ is bipartite. 
This might be one of only finitely many exceptions.

\begin{problem}
\label{problem_bipartite}
Determine the finite simple groups $S$ for which $\Lambda^+(S)$ is bipartite (resp., contains no triangles). Up to finitely many cases, do the reverse implications in \eqref{eq:implications} hold?
\end{problem}



These considerations can be viewed more generally as follows. For a noncyclic finite group $G$, consider the following invariants:

\begin{itemize}
    \item[$\diamond$] Let $\kappa(G)$ be the clique number of $\Lambda^+(G)$, that is, the largest order of a complete subgraph of $\Lambda^+(G)$.
     \item[$\diamond$] Let $\tau(G)$ be the chromatic number of $\Lambda^+(G)$, that is, the least number of colours needed to colour the vertices of $\Lambda^+(G)$ in such a way that adjacent vertices get different colours. 
     \item[$\diamond$] Let $\gamma(G)$ be the normal covering number of $G$, that is, the least number of proper subgroups of $G$ such that each element of $G$ lies in some conjugate of one of these subgroups.
\end{itemize}
The following inequalities hold:
\begin{equation}
\label{eq:inequalities}
    \kappa(G)\leq \tau(G)\leq \gamma(G). 
\end{equation}
(These are ``invariable'' versions of inequalities studied for instance in \cite{LucMar}.) The implications in \eqref{eq:implications} can be stated as follows for a general noncyclic finite group $G$:
\[\gamma(G)\leq 2\implies\tau(G)\leq 2\implies \kappa(G)\leq 2.
\]
(We note that, for every finite group $G$, $\gamma(G)\geq 2$ and, by \cite{GM} and \cite{KLS}, if $S$ is nonabelian simple then $\kappa(S)\geq 2$.) Problem \ref{problem_bipartite} asks whether, up to finitely many exceptions,
\[
\gamma(S)=2\iff \tau(S)=2 \iff \kappa(S)=2.
\]
The invariants $\kappa(G)$ and $\gamma(G)$ have been studied; see for instance Britnell--Mar\'oti \cite{BritMar}, Bubboloni--Praeger--Spiga \cite{BPS} and Garonzi--Lucchini \cite{GaronziLuc}.

As a final remark, the fact that $\Lambda(G)$ can have no triangles is somewhat strange, in comparison to classical generation. Indeed, for every $2$-generated finite group $G$ of order at least $3$, the generating graph $\Gamma(G)$ contains a triangle, and indeed ``many'' triangles. This follows from the fact that if $\gen{x,y}=G$ then $\gen{x,xy}=\gen{xy,y}=G$. The fact that this property fails for invariable generation makes it difficult to extend results from the classical to the invariable setting. See the introduction of \cite{GarLuc} for comments in this direction.

\bibliography{references}
\bibliographystyle{alpha}
\end{document}